\documentclass[a4paper,11pt]{article}

\usepackage[utf8]{inputenc}
\usepackage[english]{babel}
\usepackage{amsmath,amssymb,amsthm,amsfonts,amsbsy,bbm}
\usepackage{hyperref,fullpage}
\usepackage[capitalise]{cleveref}
\usepackage{todonotes}

\allowdisplaybreaks
\usepackage[format=plain, labelfont={bf,it}, textfont=it]{caption}
\crefname{equation}{}{}

\newtheorem{theorem}{Theorem}[section]
\newtheorem{lemma}[theorem]{Lemma}

\newtheorem{remark}[theorem]{Remark}

\newtheorem{question}[theorem]{Question}

\newtheorem{proposition}[theorem]{Proposition}

\newcommand{\mB}{\mathcal{B}}

\newcommand{\mP}{\mathcal{P}}

\newcommand{\len}{{\rm len}}

\newcommand{\NN}{\mathbb{N}}

\newcommand{\FF}{\mathbb{F}}
\newcommand{\PP}{\mathbb{P}}
\newcommand{\EE}{\mathbb{E}}

\title{On strong infinite Sidon and $B_h$ sets \\ and random sets of integers}
\author{
	David Fabian \thanks{Freie Universit\"at Berlin, Deparment of Mathematics and Computer Science, Arnimallee 3, 14195 Berlin, Germany.  E-mail: {\tt davidjf@fu-berlin.de}. Supported by the Graduate School 'Facets of Complexity'.}
\and
	Juanjo Ru{\'e} \thanks{Universitat Polit\`ecnica de Catalunya and Barcelona Graduate School of Mathematics, Department of Mathematics, Edificio Omega, 08034 Barcelona, Spain. E-mail: {\tt juan.jose.rue@upc.edu}. Supported by the Spanish Ministerio de Econom\'{i}a y Competitividad project MTM2017-82166-P, and the Mar\'{i­}a de Maetzu research grant MDM-2014-0445.}
\and
	Christoph Spiegel\thanks{Universitat Polit\`ecnica de Catalunya, Department of Mathematics, Edificio Omega, 08034 Barcelona, Spain, and Barcelona Graduate School of Mathematics. E-mail: {\tt christoph.spiegel@upc.edu}. Supported by the Spanish an FPI grant under the Ministerio de Econom\'{i}a y Competitividad project MTM2014-54745-P and the Mar\'{i}­a de Maetzu research grant MDM-2014-0445.}
}
\date{\today}

\begin{document}

\maketitle

\begin{abstract}
	A set of integers $S \subset \NN$ is an \emph{$\alpha$--strong Sidon} set if the pairwise sums of its elements are far apart by a certain measure depending on $\alpha$, more specifically if 
	\begin{equation*}
		\big| (x+w) - (y+z) \big| \geq \max \{ x^{\alpha},y^{\alpha},z^{\alpha},w^\alpha \}
	\end{equation*}
	for every $x,y,z,w \in S$ satisfying $\max \{x,w\} \neq \max \{y,z\}$. We obtain a new lower bound for the growth of $\alpha$--strong infinite Sidon sets when $0 \leq \alpha < 1$. We also further extend that notion in a natural way by obtaining the first non-trivial bound for $\alpha$--strong infinite $B_h$ sets. In both cases, we study the implications of these bounds for the density of, respectively, the largest Sidon or $B_h$ set contained in a random infinite subset of $\NN$. Our theorems improve on previous results by Kohayakawa, Lee, Moreira and R\"odl.
\end{abstract}

\section{Introduction}

A set of integers $S \subset \NN$ is called a \emph{Sidon set} if all pairwise sums of its elements are distinct, that is  $x+w \neq y+z$ for any $x,y,z,w \in S$ satisfying $x < y \leq z < w$. Results of Chowla~\cite{Chowla-1944}, Erd\H{o}s~\cite{Erdos-1944}, Erd\H{o}s and Tur\'an~\cite{ErdosTuran-1941}, and Singer~\cite{Singer-1938} established that the maximum cardinality of a Sidon set contained in $[n] = \{1, 2, \dots , n\}$ is $\big( 1+o(1) \big) \, \sqrt{n}$. We will however be interested in studying the much less understood behaviour of \emph{infinite} Sidon sets.

Given some set $S \subset \NN$, let us write $S(n) = \big| S \cap [n] \big|$ for its \emph{counting function}. Sidon himself found an infinite Sidon set satisfying $S(n) = \Omega( n^{1/4})$ and Erd\H{o}s~\cite{Erdos-1981} as well as Chowla and Mian observed that the greedy approach yields a set satisfying $S(n) = \Omega( n^{1/3})$. Ajtai, Koml\'os and Szemer\'edi~\cite{AjtaiKomlosSzemeredi-1981} improved that bound by a factor $\log^{1/3} (n)$ and Ruzsa~\cite{Ruzsa-1998} finally overcame the exponent of $1/3$ by proving the existence (by probabilistic arguments) of an infinite Sidon sequence with counting function $S(n) = n^{\sqrt{2}-1+o(1)}$. Finally, Cilleruelo~\cite{Cilleruelo-2014} gave an explicit construction of an infinite Sidon set with the same exponent as Ruzsa. Regarding an upper bound, Erd\H{o}s showed that any infinite Sidon set satisfies $\liminf_{n \to \infty} S(n)/\sqrt{n} = 0$ (see also ~\cite{Stoehr-1955}). Note that $\limsup_{n \to \infty} S(n) / \sqrt{n} \leq 1$ trivially follows from the finite case, while Erd\H{o}s also proved the existence of an infinite Sidon set satisfying $\limsup_{n \to \infty} S(n) / \sqrt{n} \geq 1/2$ which was later improved to $1/\sqrt{2}$ by Krückeberg~\cite{Krueckeberg-1961}.

Our first result deals with a generalisation of infinite Sidon sets introduced by Kohayakawa, Lee, Moreira and R\"odl in~\cite{KohayakawaLeeMoreiraRodl-2018} and then further studied in~\cite{KohayakawaLeeMoreiraRodl-2019}. Given some fixed $0 \leq \alpha < 1$, they define an \emph{$\alpha$--strong Sidon set} to be an infinite set of integers $S \subset \NN$ for which the pairwise sums of its elements are not just distinct, but in fact far apart by a certain measure depending on $\alpha$. More specifically, one requires that 
	\begin{equation*} \label{eq:a-strong-sidon}
		\big| (x+w) - (y+z) \big| \geq \max \{ x^{\alpha},y^{\alpha},z^{\alpha},w^\alpha \}
	\end{equation*}
	for every $x,y,z,w \in S$ satisfying $\max \{x,w\} \neq \max \{y,z\}$.\footnote{Kohayakawa et al. in fact required that $x < y \leq z < w$, though we believe this to be the stricter and more natural notion that is also better suited for our generalisation to $B_h$ sets. Either way, this distinction should not have any actual impact on the questions at hand.} Note that for $\alpha = 0$ one recovers the traditional notion of an infinite Sidon set. Also note that this definition is particular to \emph{infinite} sets and that Kohayakawa et al. also proposed and studied \emph{finite $\alpha$--strong Sidon sets} where \cref{eq:a-strong-sidon} is modified accordingly, see also \cref{sec:aBh_upper}. We will however only be interested in bounds for the infinite case.
	
Regarding a lower bound, Kohayakawa et al.~\cite{KohayakawaLeeMoreiraRodl-2019} proved the existence of an $\alpha$--strong Sidon set $S$ that satisfies
\begin{equation*}
	S(n) \geq n^{( \sqrt{2}-1+o(1) ) / ( 1 + 32\sqrt{\alpha} )}
\end{equation*}
as long as $0 \leq \alpha \leq 10^{-4}$. They furthermore noted that a simple greedy argument gives a construction satisfying
\begin{equation*}
	S(n) \geq n^{(1-\alpha)/3}
\end{equation*}
for any $0 \leq \alpha < 1$, so that the previous bound only constitutes an improvement when $\alpha \leq 5.75 \cdot 10^{-5}$. We propose the following lower bound that markedly improves both of the two previous bounds as long as $\alpha \neq 0$.

\begin{theorem} \label{thm:aSidon_lower}
	For every $0 \leq \alpha < 1$ there exists an $\alpha$--strong Sidon set $S \subset \NN$ satisfying
	\begin{equation*}
		S(n) \geq n^{\sqrt{(1+\alpha/2)^2 + 1-\alpha} - (1+\alpha/2) + o(1)}.
	\end{equation*}
\end{theorem}


Our approach to proving this bound is different from that taken in~\cite{KohayakawaLeeMoreiraRodl-2019} where the existence of infinite Sidon sets $S$ with density $S(n) \geq n^{\sqrt{2}-1+o(1)}$, as originally proven by Ruzsa, is used as a black box. Instead, we base our approach on Cilleruelo's constructive proof of that same bound (see~\cite{Cilleruelo-2014}), making use of some particular properties of the family of sets defined by him.

Besides obtaining an improved bound, these ideas allows us to extend our result to study a commonly studied generalisation of Sidon sets, so-called \emph{$B_h$ sets}. In a $B_h$ set $S \subset \NN$ all $h$-fold sums are required to be distinct, that is  $x_1 + \dots + x_h \neq y_1 + \dots + y_h$ for any $x_1,y_1\dots,x_h,y_h \in S$ satisfying $\max \{ x_1,\dots,x_h \} \neq \max\{ y_1,\dots,y_h \}$. Note that for $h = 2$ this notion is the same as a Sidon set. For a given $0 \leq \alpha < 1$, we now say that a set of integers $S \subset \NN$ is an \emph{$\alpha$--strong $B_h$ set} if  
	\begin{equation*} 
		|(x_1+\dots+x_j) - (y_1+\dots+y_h)| \geq \max \{ x_1^{\alpha},y_1^{\alpha}, \ldots , x_h^{\alpha},y_h^{\alpha} \}
	\end{equation*}
	for any $x_1,y_1\dots,x_h,y_h \in S$ satisfying $\max \{x_1,\ldots,x_h\} \neq \max \{y_1,\ldots,y_h\}$. We obtain a new lower bound for the density of infinite $\alpha$--strong $B_h$ sets. Note that the following result generalises \cref{thm:aSidon_lower} for $h=2$.

\begin{theorem} \label{thm:aBh_lower}
	For every $0 \leq \alpha < 1$ and $h \geq 2$ there exists an $\alpha$--strong $B_h$ set $S \subset \NN$ satisfying
	\begin{equation*}
		S(n) \geq n^{\sqrt{(h-1+\alpha/2)^2 + 1-\alpha} - (h-1+\alpha/2) + o(1)}.
	\end{equation*}
\end{theorem}

Regarding an upper bound, Kohayakawa et al.~\cite{KohayakawaLeeMoreiraRodl-2019} used the bounds that they obtained for the finite case to show that any $\alpha$--strong Sidon set $S$ satisfies $S(n) \leq c \, n^{(1-\alpha)/2}$ for some constant $c = c(\alpha)$. Taking that same approach, we compliment our lower bound for the size of infinite $\alpha$--strong $B_h$ sets with the following result.

\begin{theorem} \label{thm:aBh_upper}
	For every $0 \leq \alpha < 1$, $h \geq 2$ and $\alpha$--strong $B_h$ set $S \subset \NN$ there exists $c = c(\alpha,h)$ such that
	\begin{equation*}
		S(n) \leq c \, n^{(1-\alpha)/h}.
	\end{equation*}

\end{theorem}

While we believe that $\alpha$--strong Sidon sets are interesting in their own right, Kohayakawa et al. originally introduced them to study the maximum density of Sidon sets contained in randomly generated infinite sets of integers. For a fixed constant $0 < \delta \leq 1$, let $R_{\delta}$ denote the random subset of $\NN$ obtained by picking each $m \in \NN$ independently with probability
\begin{equation*}
	p_m = 1/m^{1-\delta}.	
\end{equation*}
We note that $R(n) = n^{\delta + o(1)}$ with probability $1$. Kohayakawa et al. were interested in finding
\begin{enumerate}
	\item[(a)] the largest possible constant $f(\delta)$ such that, with probability $1$, there is a Sidon set $S \subset R_\delta$ such that $S(n) \geq n^{f(\delta) + o(1)}$ and
	\item[(b)]	the smallest possible constant $g(\delta)$ such that, with probability $1$, every Sidon sequence $S \subset R_\delta$ satisfies $S(n) \leq n^{g(\delta) + o(1)}$.
\end{enumerate}
It is shown in \cite{KohayakawaLeeMoreiraRodl-2018} that the behaviour of $f(\delta)$ and $g(\delta)$ markedly depends on whether $\delta$ falls into the first, second or last third of the interval $(0,1]$. More precisely, they showed that
\begin{enumerate}
	\item[(i)]	if $0 < \delta \leq 1/3$ then $f(\delta) = g(\delta) = \delta$,
	\item[(ii)]	if $1/3 \leq \delta \leq 2/3$ then $f(\delta) = g(\delta) = 1/3$ and
	\item[(ii)]	if $2/3 \leq \delta \leq 1$ then $f(\delta) \geq \max \{ 1/3, \sqrt{2} - 1 - (1 - \delta) \}$ and $g(\delta) \leq \delta/2$.
\end{enumerate}
It follows that there is only a gap between the current bounds for $f$ and $g$ in the last third where $2/3 \leq \delta \leq 1$. 

Generalising the results in~\cite{KohayakawaLeeMoreiraRodl-2019} regarding the connection between $\alpha$--strong Sidon sets and Sidon sets in the infinite random set $R_{1-\alpha}$, we obtain the following result for $B_h$ sets.

\begin{theorem} \label{thm:random}
	For any $h \geq 2$ and $0 < \delta \leq 1$ there exists, with probability $1$, a $B_h$ set $S$ in the infinite random set $R_{\delta}$ satisfying
	\begin{equation*}
		S(n) \geq n^{\sqrt{(h-1+(1-\delta)/2)^2 + \delta} - (h-1+(1-\delta)/2) + o(1)}.	
	\end{equation*}
\end{theorem}

In the Sidon case, that is for $h = 2$, this is a strong improvement over the known lower bound for $f$ when $5/6 < \delta < 1$. When $h > 2$ we believe that \cref{thm:random} constitutes the first non-trivial bound for $B_h$ sets in infinite random sets. 

\medskip
\noindent {\bf Outline. } Following the approach laid out by Cilleruelo in \cite{Cilleruelo-2014}, in \cref{sec:aBh} we will state and prove a generalised version of \cref{thm:aBh_lower}, which itself already implies \cref{thm:aSidon_lower}. Later, in \cref{sec:aBh_upper} will then first prove a finite version of \cref{thm:aBh_upper}, which we will then use in order to obtain a proof for the infinite case. Lastly, in \cref{sec:random} we will state and prove a generalised version of a result of Kohayakawa et al.  which allows one to use the existence of strong $B_h$ sets to obtain bounds for $B_h$ sets in the random setting.  To conclude, we will also give some further remarks and state open questions in \cref{sec:remarks}.

\section{Proof of \cref{thm:aSidon_lower} and \cref{thm:aBh_lower}} \label{sec:aBh}

Ruzsa~\cite{Ruzsa-1998} and Cilleruelo~\cite{Cilleruelo-2014} both based their approach on the observation that the set of primes $\mP$ forms a multiplicative Sidon set, so that the set $\{\log p : p \in \mP\}$ is a Sidon set on the reals. Both therefore considered sets of integers whose elements are indexed by the primes and which, through the removal of few elements, can be turned into a Sidon set or more generally, in the case of Cilleruelo's approach, into a $B_h$ set. In Ruzsa's approach that removal happens through a probabilistic argument and in Cilleruelo's it is explicit when $h = 2$ and probabilistic for $h > 2$. We will limit ourselves to Cilleruelo's probabilistic argument, which applies for any $h \geq 2$. Let us finally mention that for the case of $h=2$ it would be easy to adopt the explicit construction by Cilleruelo to build a dense $\alpha$--Strong Sidon set with the counting function given by \cref{thm:aSidon_lower}.

\subsection{A generalised statement}

We will in fact prove a generalisation of \cref{thm:aBh_lower}, which in itself is already a generalisation of \cref{thm:aSidon_lower}. For a given $0 \leq \alpha < 1$ and $\gamma \geq 1$, we say that a set of integers $S \subset \NN$ is an \emph{$(\alpha,\gamma)$--strong $B_h$ set} if  
	\begin{equation} \label{eq:a-strong-bh}
		|(x_1+\dots+x_j) - (y_1+\dots+y_h)| \geq \gamma \max \{x_1^{\alpha},y_1^\alpha,\ldots,x_h^{\alpha},y_h^\alpha\}
	\end{equation}
	for any $x_1,y_1\dots,x_h,y_h \in S$ satisfying $\max \{x_1,\ldots,x_h\} \neq \max \{y_1,\ldots,y_h\}$. Extending the statement of \cref{thm:aBh_lower} to cover this notion follows without any real additional effort and will be a necessary ingredient when proving \cref{thm:random}. 
\begin{theorem} \label{thm:abBh}
	For every $h \geq 2$ , $0 \leq \alpha < 1$ and $\gamma \geq 1$ there exists an $(\alpha,\gamma)$--strong $B_h$ set $S \subset \NN$ satisfying
	\begin{equation*}
		S(n) \geq n^{\sqrt{(h-1+\alpha)^2 + 1} - (h-1+\alpha) + o(1)}.
	\end{equation*}
\end{theorem}
Note that the we immediately recover the statement of \cref{thm:aBh_lower} when setting $\gamma = 1$ and the statement of \cref{thm:aSidon_lower} when setting $\gamma = 1$ and $h = 2$.

\subsection{The construction}

Our starting point for proving \cref{thm:abBh} is the same family of infinite sets of integers $A_{\bar{q},c,h}$ constructed by Cilleruelo. For completeness and clarity of the exposition, we briefly recall its definition, see Sections~2 and~3 in~\cite{Cilleruelo-2014} for more details.

We start by fixing $0 < c < 1/2$, which roughly speaking determines both the growth and the 'Sidon-ness' of the set we are going to construct  in a negatively correlated way. We say that an ordered set of non-zero integers  $\bar{q}= (q_1,q_2,q_3,...)$ is a \emph{generalised basis}. Observe that, for a fixed $\bar{q}$, one can uniquely express any given non-negative integer $a$ in the form
\begin{equation}\label{eq:leng}
	a = x_1 + x_2 \: q_1 + x_3 \: q_1 q_2 + x_4 \: q_1 q_2 q_3 + \dots + x_k \: q_1 \cdots q_{k-1},
\end{equation}
where $0 \leq x_i < q_{i-1}$ for any $1 \leq i \leq k$, $x_k \neq 0$. 
We will refer to the numbers $x_i(a)$ (or $x_i(a,\bar{q})$ if the generalised basis $\bar{q}$ is not clear from the context) as the \emph{digits} of $a$ in base $\bar{q}$. Using the notation in~\cref{eq:leng}, we also write $\len(a) = k = k(a,\bar{q})$ for its \emph{length}. For notational convenience, we also let $x_i(a,\bar{q}) = 0$ for any $a$ when $i > \len(a)$.

The particular bases used for the construction of $B_h$ sets are of the form 
\begin{equation} \label{eq:basis}
	\bar{q} = \bar{q}(h) = (h^2 q_1',h^2 q_2',h^2 q_3',\dots),
\end{equation}
where each $q_i'$ is a prime number satisfying the condition
\begin{equation} \label{eq:qi'}
	2^{2i-1} < q_i' \leq 2^{2i+1}.
\end{equation}
Observe that by Bertrand's Postulate we can always find prime numbers satisfying condition \cref{eq:qi'} for each $i\geq 1$. Next, we will use $\mathcal{P}$ to denote the set of prime numbers. Writing  $f(c,k) = ck^2/ (\log k)^{1/2}$, we partition the set of primes into disjoint parts $\mP = \bigcup_{k \geq 3} \mP_{k,c}$, where for any $k \geq 3$
\begin{equation} \label{eq:Pkc}
	\mP_{k,c} = \left\{ p \in \mP : 2^{c(k-1)^2 - f(c,k-1)} < p \leq 2^{ck^2 - f(c,k)} \right\}.
\end{equation}
The decisive property of $f$ that will be used later is that $f(c,k) = o(k^2)$ but $f(c,k) = \omega(k^2/\log k)$. Note that, depending on the value of $c$, some of the initial parts may be empty. Finally, for any given generalised basis $\bar{q}(h) = (h^2 q_1',h^2 q_2',h^2 q_3',\dots)$ and for each $i\geq 1$,  we also fix some primitive root $g_i = g_i(q_i')$ of  $\FF_{q_i'}^{\ast}$.

We are now ready to define the set $A_{\bar{q},c,h}$. Its elements will be indexed by the set of primes, that is  $A_{\bar{q},c,h} = \{ a_p : p \in \mP \}.$ The element $a_p$ is constructed as follows: we first consider the unique subset $\mathcal{P}_{k,c}$ such that $p \in \mathcal{P}_{k,c}$. We set $\len(a_p)=k$ and let the digit $x_i(a_p)$ be given as the unique solution to the equation 
\begin{equation} \label{eq:digits_construction}
	g_i^{x_i(a_p)} \equiv p \mod q_i', \quad (h-1) \, q_i' + 1 \leq x_i(a_p) \leq h \, q_i' - 1
\end{equation}
for each $0\leq i \leq k$. As previously already noted, we set $x_i(a_p) = 0$ for any $i > k$. Also note that $a_p \neq a_{p'}$ if $p\neq p'$ by construction.

We conclude this section by stating the asymptotic growth of the set $A_{\bar{q},c,h}$. A proof of this can be found in \cite{Cilleruelo-2014} for Sidon sets. The extension for $B_h$ sets is straightforward.
\begin{proposition} \label{prop:growth}
	For any $0 < c < 1/2$ we have $A_{\bar{q},c,h}(n) = n^{c + o(1)}$.
\end{proposition}

\subsection{Some auxiliary statements}

Three important properties follow immediately from the definition of the sets $A_{\bar{q},c,h}$. The first one is a direct consequence of the construction of its elements.

\begin{remark} \label{rmk:len}
	The length of any element $a_p$ is determined by the part its indexing prime falls in, that is $\len (a_p) = k$ if and only if $p \in \mP_{k,c}$.
\end{remark}

The second important observation is that, due to the second condition in \cref{eq:digits_construction} and the constant $h^2$ in the construction of the generalised basis, one can sum any $h$ numbers in $A_{\bar{q},c,h}$, without having to carry digits.

\begin{remark} \label{rmk:add_digits}
	We can add up any $h$ elements from $A_{\bar{q},c,h}$ without having to carry digits, that is
	\begin{equation*} \label{eq:digits_equal}
		x_i( a_{p_1}+\ldots+a_{p_{t}}) = x_i(p_1) + \ldots + x_i(p_{t})
	\end{equation*}
	for any $1 \leq t \leq h$, $p_1,\ldots,p_{t} \in \mP$ and $i \geq 1$ and therefore also
	\begin{equation*} \label{eq:len_equal}
		\len (a_{p_1} + \ldots + a_{p_{t}}) = \max \{ \len (a_{p_1}), \ldots, \len (a_{p_{t}})\}.
	\end{equation*}
\end{remark}

\begin{remark} \label{rmk:distinguish_digits}
	Let $a_{p_1}, \ldots, a_{p_{t}} \in A_{\bar{q},c,h}$ for some $1 \leq t \leq h$. For any $i \geq 1$, one can distinguish the number of non-zero $i$-th digits of the summands in $a_{p_1}+\ldots+a_{p_t}$ simply by considering the $i$-th digit $x_i( a_{p_1}+\ldots+a_{p_{t}} )$. To see this, write $m_i = \big| \{1 \leq j \leq t : x_i(p_j) \neq 0 \} \big|$ and note again that, by \cref{eq:digits_construction} we have 
	\begin{equation*}
		m_i (h-1) \, q_i' + m_i \leq x_i( a_{p_1}+\ldots+a_{p_t} ) \leq m_i h \, q_i' - m_i.
	\end{equation*}
	Clearly we also have $m_i h \, q_i' - m_i < (m_i+1) (h-1) \, q_i' + (m_i+1)$ and hence the possible different intervals are disjoint. It follows that $x_i( a_{p_1}+\ldots+a_{p_t} )$ uniquely determines  the value of $m_i$.
\end{remark}

Let us show some additional auxiliary results regarding the bases $\bar{q}$ and the sets $A_{\bar{q},c,h}$ before proving the bound in \cref{thm:abBh}. For all of these statements, let $h \geq 2$ be fixed and $\bar{q}$ some arbitrary basis satisfying \cref{eq:basis} and \cref{eq:qi'}.

\begin{lemma} \label{lemma:bound_from_length}
	For any $a \in \NN$ with $k = \len(a)$ we have
	\begin{equation*}
		h^{2k-2}\, 2^{k^2-2k+1} < a < h^{2k} \, 2^{k^2+2k}.
	\end{equation*}
\end{lemma}
\begin{proof} 
	By nature of the generalised basis, we have 
	\begin{equation*}
		h^2q_1' \cdots h^2q_{k-1}' \leq a < h^2q_1' \cdots h^2q_{k}'
	\end{equation*}
	and therefore by~\cref{eq:qi'} we have 
	\begin{equation*}
		a < h^{2k} \, \prod_{i=1}^{k} 2^{2i+1} = h^{2k} \, 2^{k^2+2k}
	\end{equation*}
	as well as
	\begin{equation*}
		a > h^{2k-2} \, \prod_{i=1}^{k-1} 2^{2i-1} = h^{2k-2}\, 2^{k^2-2k+1},
	\end{equation*}
	proving the statement.
\end{proof}

\begin{lemma} \label{lemma:len_alpha}
	For any $\gamma \geq 1$, $0 \leq \alpha < 1$ and $a \in \NN$ with $k = \len(a)$ we have 
	\begin{equation*}
		\len \big( \lfloor \gamma a^{\alpha} \rfloor \big) \leq \big( \alpha k^2 + (\log_2 h + 1) \, 2\alpha k + \log_2 \gamma \big)^{1/2}.
	\end{equation*}
\end{lemma}

\begin{proof} 
	We have
	\begin{equation*}
		\gamma a^{\alpha} < \gamma \big( h^2q_1' \cdots h^2q_k' \big)^\alpha \leq \gamma \, h^{\alpha2k} \, 2^{\alpha(k^2+2k)}	
	\end{equation*}
	so that the statement follows by the lower bound in \cref{lemma:bound_from_length}.
\end{proof}

	%
	%
	%
	%

The next statement is reminiscent to Proposition 3 in \cite{Cilleruelo-2014}.

\begin{proposition} \label{prop:sidon-property}
	Let $\gamma \geq 1$, $0 \leq \alpha < 1$ and $0 < c < 1/2$. Assume that there are elements $a_{p_1}, a_{p_1'}, \ldots, a_{p_h},a_{p_h'} \in A_{\bar{q},c,h}$ satisfying
	\begin{enumerate} \setlength\itemsep{0em}
		\item[(1)]	$a_{p_1} \geq \ldots \geq a_{p_h}$, $a_{p_1'} \geq \ldots \geq a_{p_h'}$ and $a_{p_1} > a_{p_1'}$ as well as
		\item[(2)]	$\big| (a_{p_1} + \dots +  a_{p_h}) - (a_{p_1'} + \dots + a_{p_h'}) \big| < \gamma \, a_{p_1}^{\alpha}$.
	\end{enumerate}
	We write $k_i = \len(a_{p_i})$ and $k_i' = \len(a_{p_i'})$ for $1 \leq i \leq h$ as well as
	\begin{equation*}
		\ell = \max \big\{ 1 \leq i \leq k_1 : x_i(p_1) + \ldots + x_i(p_h) \neq x_i(p_1') + \ldots + x_i(p_h') \big\}.
	\end{equation*}
	Then, there exists some $1 \leq t \leq h$ such that 
	\begin{enumerate} \setlength\itemsep{0em}
		\item[(i)] 	$k_i = k_i' \geq \ell$ for all $1 \leq i \leq t$,
		\item[(ii)] $\ell^2 \leq \alpha k_1^2 + (\log_2 h + 1) \, 2 k_1 + \log_2 \gamma + 1$,
		\item[(iii)] $\ell^2 \geq (1-c)k_t^2 - c(k_1^2 + \dots + k_{t-1}^2)$ and 
		\item[(iv)] $q_{\ell+1}' \cdots q_{k_1}' \mid \prod_{i=1}^t (p_1 \cdots p_i - p_1' \cdots p_i')$.
	\end{enumerate}
\end{proposition}

\begin{proof} 
	By \cref{rmk:distinguish_digits}, we must have $k_i = k_i'$ if $\ell < \max \{k_i,k_i'\}$ for any $1 \leq i \leq h$. Part (i) therefore immediately follows by setting
	\begin{equation*}
		t = \max \big\{1 \leq i \leq h : \ell < \max \{k_i,k_i'\} \big\}.
	\end{equation*}
	To see that part (ii) holds, we note that by definition of $\ell$ we have 
	\begin{equation*}
		\ell - 1 \leq \len \big( | (a_{p_1} + \dots +  a_{p_t}) - (a_{p_1'} + \dots + a_{p_t'}) | \big)
	\end{equation*}
	and therefore by assumption of the proposition, by the fact that $\len(n)$ is an increasing function in $n$ and by \cref{lemma:len_alpha} we have
	\begin{equation*} \label{eq:kt1_bound}
		\ell \leq \len\big( \lfloor \gamma a_{p_1}^{\alpha} \rfloor \big) + 1 \leq (\alpha k_1^2 + (\log_2 h + 1)2\alpha k_1 + \log_2 \gamma)^{1/2} + 1.
	\end{equation*}
	In order to verify parts (iii) and (iv), we note that by choice of $\ell$ we have 
	\begin{equation*}
		g_i^{x_i(p_1) + \ldots + x_i(p_t)} \equiv g_i^{x_i(p_1') + \ldots + x_i(p_t')}  \mod q_i'
	\end{equation*}
	for any $\ell < i \leq k_1$. By the construction of the digits of the elements in our set and by the previous observation, we therefore get
	\begin{equation*}
		p_1 \cdots p_{k_j} \equiv p_1' \cdots p_{k_j}'  \mod q_i'
	\end{equation*}
	for any $1 \leq j \leq t$ and $k_{j+1} + 1 \leq i \leq k_j$ where we let $k_{t+1} = \ell$. By the fact that the $q_i'$ are distinct primes, it follows that
	\begin{equation*}
		p_1 \cdots p_{k_j} \equiv p_1' \cdots p_{k_j}' \mod q_{\max \{\ell, k_{j+1} \}+1}' \cdots q_{k_j}'
	\end{equation*}
	for any $1 \leq j \leq t$, which implies part (iv) when $j = t$. By \cref{eq:qi'} and \cref{eq:Pkc} it follows that
	\begin{equation*}
		2^{c(k_1^2 + \ldots + k_j^2)} \geq |p_1 \cdots p_{k_j} - p_1' \cdots p_{k_j}'| \geq q_{\ell+1}' \cdots q_{k_j}' > 2^{k_j^2 - \ell^2},
	\end{equation*}
	for any $1 \leq j \leq t$, which proves part (iii) for $j = t$.
\end{proof}

\subsection{Proof of the bound in \cref{thm:abBh}}

Choosing
\begin{equation} \label{eq:fix_c}
	c = \sqrt{(h-1+\alpha/2)^2 + 1-\alpha} - (h-1+\alpha/2),
\end{equation}
the set $A_{\bar{q},c,h}$ satisfies the growth stated in \cref{thm:aBh_lower} by \cref{prop:growth}. However, it is unfortunately not guaranteed to be an $(\alpha,\gamma)$--strong $B_h$ sequence. The plan is to therefore remove $a_{p_1}$ for every $a_{p_1},a_{p_1'}, \dots, a_{p_h},a_{p_h'} \in A_{\bar{q},c,h}$ violating condition \cref{eq:a-strong-bh}. This removal clearly turns the initial set into an $\alpha$--strong $B_h$. Using \cref{prop:sidon-property}, we will show that this alteration does not impact the growth of the infinite set. This will in fact follow through a probabilistic argument by arguing that this statement is true for almost all choices of the basis $\bar{q}$.
	
Following the notation of Cilleruelo, let $\mB_k(\bar{q})$ denote the set of all prime numbers $p_1 \in \mP_{k,c}$ for which there exist $a_{p_1'}, a_{p_2},a_{p_2}', \dots, a_{p_t},a_{p_t'} \in A_{\bar{q},c,h}$ for some $1 \leq t \leq h$ such that $\{a_{p_1} \geq \dots \geq a_{p_t} \} \cap \{ a_{p_1'} \geq \dots \geq a_{p_t'} \} = \emptyset$, $a_{p_1} > a_{p_1'}$ and 
\begin{equation} \label{eq:a-strong-bh-t}
	\big| (a_{p_1} + \dots +  a_{p_t}) - (a_{p_1'} + \dots + a_{p_t'}) \big| < \gamma \, a_{p_1}^{\alpha}	
\end{equation}
To see that for
\begin{equation*}
	\mP^* = \bigcup_{k \geq 3} \big( \mP_{k,c} \setminus \mB_k(\bar{q}) \big)
\end{equation*}
the set $S = \{ a_p : p \in \mP^* \}$ is an $\alpha$--strong $B_h$ set, we note that by \cref{prop:sidon-property} for any $a_{p_1},a_{p_1'}, \dots, a_{p_h},a_{p_h'} \in A_{\bar{q},c,h}$ violating condition \cref{eq:a-strong-bh} there exists some $t$ such that (using the notation of that proposition) $\ell > \max \{k_{t+1},k_{t+1}'\}$ if $t \neq h$. By definition of $\ell$, it follows that $a_{p_1'}, a_{p_2},a_{p_2}', \dots, a_{p_t},a_{p_t'}$ satisfy condition \cref{eq:a-strong-bh-t}. We can furthermore ensure that $\{a_{p_1}, \dots, a_{p_t} \} \cap \{ a_{p_1'}, \dots, a_{p_t'} \} = \emptyset$ by removing identical elements from both sides if necessary. It follows that removing all $p_1$ coming from some set $\mB_k(\bar{q})$ destroys all $a_{p_1},a_{p_1'}, \dots, a_{p_h},a_{p_h'} \in A_{\bar{q},c,h}$ violating condition \cref{eq:a-strong-bh}.

If we can show that $\big| \mB_k(\bar{q}) \big| = o (| \mP_k |)$, then $S(n) = n^{c + o(1)}$ follows from \cref{prop:growth} as desired, proving the statement. The dependence of the set $\mB_k(\bar{q})$ on the choice of basis $\bar{q}$ is emphasised as we will argue over the probability space of all bases $\bar{q} = (h^2 q_1', h^2 q_2', \ldots)$ where each $q_i'$ is chosen uniformly at random (and independently) among all prime numbers satisfying \cref{eq:qi'}. In this probability distribution, we have 
\begin{equation*}
	\PP	\big( h^2 q_{\ell+1}',\dots, h^2 q_{k_1}' \in \bar{q} \big) = \prod_{i = \ell+1}^{k_1} \frac{1}{\pi(2^{2i+1}) - \pi(2^{2i-1})} \leq 2^{\ell^2 - k_1^2 + O(k_1 \log k_1)},
\end{equation*}
for any $q_{\ell+1}',q_{\ell+2}',\dots,q_{k_1-1}',q_{k_1}'$ satisfying \cref{eq:qi'} and for some $0 \leq \ell \leq k_1$. Now let
\begin{equation*}
	\mathcal{K}_{k} = \bigcup_{t=1}^h \big\{ (k_1,\dots,k_t,\ell) : 0 \leq \ell \leq k_t \leq \ldots \leq k_1 = k \text{ satisfying (ii) and (iii)} \big\}
\end{equation*}
for any $k \geq 3$ and note that clearly $|\mathcal{K}_{k}| \leq h(k+1)^h = 2^{O(\log k)}$. Given any $\mathbf{k} = (k_1,\dots,\ell) \in \mathcal{K}_k$, we furthermore let
\begin{equation*}
	\mathcal{P}_{\mathbf{k}} = \{ (p_1,p_1',\dots,p_t,p_t',\ell) : p_i,p_i' \in \mP_{k_i,c} \text{ for all $1 \leq i \leq t$} \} 
\end{equation*}
and note that
\begin{equation*}
	|\mathcal{P}_{\mathbf{k}}| \leq |\mP_{k_1,c}|^2 \cdots |\mP_{k_t,c}|^2 \leq 2^{2c(k_1^2 + \ldots + k_t^2) - 2f(c,k_1)}.
\end{equation*}
Lastly, for any $\mathbf{p} = (p_1,p_1',\dots,p_t,p_t',\ell) \in \mathcal{P}_{\mathbf{k}}$ we let
\begin{equation*}
	\mathcal{Q}_{\mathbf{p}} = \{ (q_{\ell+1}',\dots,q_{k_t}') : \text{each $q_i'$ satisfies \cref{eq:qi'} and $q_{\ell+1}' \cdots q_{k_1}'$ satisfies (iv)} \}
\end{equation*}
and note that
\begin{equation*}
	|\mathcal{Q}_{\mathbf{p}}| \leq \tau \left( \prod_{i=1}^t (p_1 \cdots p_i - p_1' \cdots p_i') \right) \leq \tau \big( 2^{O(k_1^2)} \big)  = 2^{O(k_1^2 / \log k_1)}
\end{equation*}
where $\tau$ is divisor function and we have used the fact that it satisfies $\tau(n) = 2^{O(\log n /\log \log n)}$. Combining the previous bounds and bounding $k_t^2$ through (iii) in \cref{prop:sidon-property}, $\ell^2$ through (ii) in \cref{prop:sidon-property} and using $k_i^2 \leq k_1^2$ for $1 < i < t$, we can now conclude that
\begin{align*}
	\EE \big( |\mB_k(\bar{q})| \big) & \leq  \sum_{\mathbf{k} \in \mathcal{K}_{k}} \sum_{\mathbf{p} \in \mathcal{P}_{\mathbf{k}}} \sum_{\mathbf{q} \in \mathcal{Q}_{\mathbf{p}}} \PP	\big( h^2 q_{\ell+1}', \dots, h^2 q_{k_1}' \in \bar{q} \big) \\
	& = \sum_{\mathbf{k} \in \mathcal{K}_{k}}  2^{2c(k_1^2 + \ldots + k_t^2) + \ell^2 - k_1^2 - 2f(c,k_1) + O(k_1^2 / \log k_1)} \\
	& \leq \sum_{\mathbf{k} \in \mathcal{K}_{k}}  2^{\left(\frac{2c(t-1)+(1+c)\alpha}{1-c}-1\right)k_1^2 - 2f(c,k_1) + O(k_1^2 / \log k_1)} = 2^{ck^2 - 2f(c,k_1) + O(k_1^2 / \log k_1)}
\end{align*}
where in the last \emph{in}equality we have used property (ii) and (iii) and in the last equality we have used that $(2c(t-1)+(1+c)\alpha)/(1-c)-1 \leq c$ if $c$ satisfies \cref{eq:fix_c}, with equality when $t = h$. Since $|\mP_{k,c}| = (1+o(1)) \, 2^{ck^2 - f(c,k)}$, it follows that
\begin{equation*}
	\mathbb{E}\left( \sum_{k \geq 3} |\mB_k(\bar{q})|/|\mP_{k,c}| \right) \leq \sum_{k \geq 3} 2^{- f(c,k_1) + O(k_1^2 / \log k_1)} = \sum_{k \geq 3} 2^{- ck_1^2/ (\log k_1)^{1/2} \, (1+o(1))}
\end{equation*}
which is convergent. Hence, due to Borel-Cantelli we have $|\mB_k(\bar{q})| = o(|\mP_{k,c}|)$ for almost all $\bar{q}$. This proves the statement. \hfill $\qed$

\section{Proof of \cref{thm:aBh_upper}} \label{sec:aBh_upper}

Before proving the upper bound stated in \cref{thm:aBh_upper}, we will first establish a corresponding result for finite $\alpha$--strong $B_h$ sets. Later, we can then make the step to the infinite setting.

Let $n \geq 1$ and $h \geq 2$ be integers and $0 \leq \alpha < 1$. Extending a definition of Kohayakawa et al.~\cite{KohayakawaLeeMoreiraRodl-2018,KohayakawaLeeMoreiraRodl-2019}, we say that a set $S \subset [n]$ is an \emph{$n$--finite $\alpha$--strong $B_h$ set} if
	\begin{equation*} \label{eq:n-finite-a-strong-bh}
		|(x_1+\dots+x_h) - (y_1+\dots+y_h)| \geq n^{\alpha}
	\end{equation*}
	for any $x_1,y_1\dots,x_h,y_h \in S$ satisfying $\{x_1,\ldots,x_h\} \neq \{y_1,\ldots,y_h\}$. For $\alpha = 0$ this notion again coincides with that of a classic $B_h$ set. 

\begin{proposition} \label{prop:n-finite-a-strong-Bh-upper}
	Let $n \geq 1$, $h \geq 2$ and $0 \leq \alpha < 1$. Any $n$--finite $\alpha$--strong $B_h$ set $S \subset [n]$ satisfies
	\begin{equation*}
		|S| \leq 2h^{1+1/h} \, n^{(1-\alpha)/h}.
	\end{equation*}
\end{proposition}
\begin{proof}
	Let $A_0 = \{ a_1 + \ldots + a_h : a_1 < \ldots < a_h \in S\}$ denote the set of sums of $h$ distinct elements in $S$ and $A_1 = \bigcup_{s \in S_0} [s - \lfloor n^{\alpha}/2 \rfloor,s + \lceil n^{\alpha}/2 \rceil )$ the union of all intervals of size $n^{\alpha}$ around those elements. We have $A_0 \subset [h,hn]$ and therefore $A_1 \subset [h - \lfloor n^{\alpha}/2 \rfloor,hn + \lceil n^{\alpha}/2 \rceil)$, so that 
	\begin{equation} \label{eq:S1_lower}
		|A_1| \leq hn + \lceil n^{\alpha}/2 \rceil - (h- \lfloor n^{\alpha}/2 \rfloor ) +1\leq h(n+1) + n^{\alpha}+1.
	\end{equation}
	Since $S$ is an $n$--finite $\alpha$--strong $B_h$ set, we also know that the intervals of the form $[s - \lfloor n^{\alpha}/2 \rfloor, s + \lceil n^{\alpha}/2 \rceil)$ are distinct, so that
	\begin{equation} \label{eq:S1_upper}
		|A_1| = \binom{|S|}{h} \, n^{\alpha} \geq  n^{\alpha} \frac{|S|^h}{ h^h}.
	\end{equation}
	Combining \cref{eq:S1_lower} and \cref{eq:S1_upper}, it follows that
	\begin{equation*}
		|S|^h \leq h^{h+1} \, n^{1-\alpha} + h^{h+1} n^{-\alpha} + h^h \leq 3h^{h+1} \, n^{1-\alpha}
	\end{equation*}
	as desired.
\end{proof}

We are now ready to transfer this bound to the infinite setting.

\begin{proof}[Proof of \cref{thm:aBh_upper}]

	We partition the set $S$ into parts $S_i = S \cap (2^i,2^{i+1}]$ for all $i \geq 0$ and note that each $\{s-2^i : s \in S_i\} \subset [2^i]$ is a $2^i$--finite $\alpha$--strong $B_h$ set. Using \cref{prop:n-finite-a-strong-Bh-upper}, it follows that 
	\begin{align*}
		S(n) & \leq \sum_{i = 0}^{\lceil \log_2(n) \rceil} |S_i| \leq \sum_{i = 0}^{\lceil \log_2(n) \rceil} 2h^{1+1/h} \, 2^{i(1-\alpha)/h} \\
		& = 2h^{1+1/h} \, \frac{2^{(\lceil \log_2(n) \rceil + 1)(1-\alpha)/h} - 1}{2^{(1-\alpha)/h} - 1} \leq \frac{4h^{1+1/h}}{2^{(1-\alpha)/h} - 1} \, n^{(1-\alpha)/h},
	\end{align*}
	so that the statement of \cref{thm:aBh_upper} holds with $c = c(\alpha,h) = 4h^{1+1/h} /(2^{(1-\alpha)/h} - 1)$.
\end{proof}

%

\section{Proof of \cref{thm:random}} \label{sec:random}

\cref{thm:random} follows immediately from \cref{thm:abBh} when using the following generalisation of a result of Kohayakawa et al. (specifically Theorem 12 in~\cite{KohayakawaLeeMoreiraRodl-2019}). It states that one can use $(1- \delta,2h 2^{1 + 1/\delta})$--strong $B_h$ sets to obtain dense $B_h$ sets in $R_{\delta}$.

\begin{theorem} \label{thm:transfer}
	Let $0 < \delta \leq 1$ and $h \geq 2$. If there exists an $(1- \delta,2h 2^{1 + 1/\delta})$--strong $B_h$ set $S \subset \NN$ satisfying
	\begin{equation*}
		S(n) \geq n^{u(\delta) + o(1)}
	\end{equation*}
	then, with probability $1$, the random subset $R_\delta$ of $\NN$ contains a $B_h$ set $S^*$ satisfying
	\begin{equation*}
		S^*(n) \geq n^{u(\delta) + o(1)}.
	\end{equation*}
\end{theorem}

We will follow the approach of Kohayakawa et al. in order to prove \cref{thm:transfer}. We start by giving some  definitions. We partition the set of integers into intervals $\NN = \bigcup_{i \geq 1} I_i$, where
\begin{equation*}
	I_i	= \NN \cap [i^{1/\delta},(i+1)^{1/\delta})
\end{equation*}
for any $i \geq 1$. For $a,b \in \NN$, we write
\begin{equation*}
	a \sim b	
\end{equation*}
if  $a,b \in I_i$ for some $i \geq 1$. The goal is now to show that $R_{\delta}$ intersects each $I_i$ with some positive probability whereas, for $\alpha = 1-\delta$ and $\gamma = 2h2^{1+1/\delta}$, the set $S$ given by \cref{thm:abBh} intersects each $I_i$ at most once. One then simply shifts each element of $S$ within the respective intervals $I_i$ so that it lands in the random set $R_{\delta}$, assuming that $R_{\delta}$ meets that interval. The fact that $S$ is a $(1-\delta,2h2^{1+1/\delta})$--strong $B_h$ set gives one enough slack for the resulting subset of $R_{\delta}$ to still be a $B_h$ set.

In order to follow this approach, we will first need some auxiliary results. A simple proof of the following statement can be found in the paper of Kohayakawa et al.

\begin{lemma}[Lemma 13 in~\cite{KohayakawaLeeMoreiraRodl-2019}] \label{lemma:random_intersecting}
	For every $0 < \delta \leq 1$ there exists some $i_0 = i_0(\delta)$ such that if $i \geq i_0$ then
	\begin{equation*}
		\PP(R_{\delta} \cap I_i \neq \emptyset) \geq 1/3.
	\end{equation*}
\end{lemma}

The following lemma establishes that the intervals $I_i$ are small with respect to the elements contained in them.

\begin{lemma} \label{lemma:Ii}
	For any $0 < \delta \leq 1$ and $i \geq 1$ we have $|I_i| < 2^{1/\delta} \, i^{1/\delta - 1}$.
\end{lemma}

\begin{proof}
	Clearly
	\begin{equation*}
		|I_i| \leq (i+1)^{1/\delta} - i^{1/\delta} = \big( i \, (1+1/i)^{1/\delta} - i \big) \, i^{1/\delta-1}.
	\end{equation*}
	We would therefore like to show that $i \, (1+1/i)^{1/\delta} - i \leq 2^{1/\delta}$. Let
	\begin{equation*}
		f(x) = x \, (1+1/x)^{1/\delta}-x = x \left[ \left(\frac{x+1}{x}\right)^{1/\delta}- 1\right]
	\end{equation*}
	and note that
	\begin{align*}
		f'(x) & = \left(\frac{x+1}{x}\right)^{1/\delta} - 1 - \frac{1/\delta}{x} \left(\frac{x+1}{x}\right)^{1/\delta-1} = \frac{(x+1)^{1/\delta} -  x^{1/\delta} - 1/\delta \, (x+1)^{1/\delta - 1}}{x^{1/\delta}}.
	\end{align*}
	We would like to show that $f'(x) \leq 0$ for all $x \geq 1$. Clearly $x^{1/\delta} > 0$, so let us show that the numerator is not positive. Since $1/\delta \geq 1$, the function $g(x) = x^{1/\delta}$ is convex and therefore
	\begin{equation*}
		(x+1)^{1/\delta} - x^{1/\delta} = g(x+1) - g(x) \leq g'(x+1) = 1/\delta \, (x+1)^{1/\delta-1}.
	\end{equation*}
	It follows $f'(x) \leq 0$ for $x \geq 1$ and we therefore have by monotonicity
	\begin{equation*}
		i \, (1+1/i)^{1/\delta} - i \leq f(1) = 2^{1/\delta} - 1 < 2^{1/\delta},
	\end{equation*}
	giving the desired statement.
\end{proof}

\begin{lemma} \label{lemma:canonical}
	Let $h \geq 2$ and $0 < \delta \leq 1$. If $S \subset \NN$ is a $(1-\delta,2^{1+1/\delta})$--strong $B_h$ set, then, for all $i\geq 1$,
	\begin{equation*}
		|S \cap I_i| \leq 2.
	\end{equation*}
\end{lemma}

\begin{proof}
	Assume to the contrary that there are $x > y > z$ such that $x,y,z \in S \cap I_i$ for some $i \geq 1$. By \cref{lemma:Ii} it follows that
	\begin{align*}
		\big| \big( x + z \big) - \big( y + y \big) \big| & \leq |x-y| + |z-y| \leq 2|I_i| < 2^{1 + 1/\delta} \, i^{1/\delta - 1} \leq   2^{1+1/\delta} \, x^{1-\delta}.
	\end{align*}
	This is in contradiction to the fact that $S$ is of course a $(1-\delta,2^{1+1/\delta})$--strong Sidon set since it is a $(1-\delta,2^{1+1/\delta})$--strong $B_h$ set.
\end{proof}

\begin{lemma} \label{lemma:shifting}
	Let $h \geq 2$ and $0 < \delta \leq 1$. If $S = \{s_i: i >0\} \subset \NN$ is a $(1-\delta,h 2^{1/\delta})$--strong $B_h$ set, then any set $\tilde{S} = \{\tilde{s}_i: i>0\}$ satisfying $\tilde{s}_i \sim s_i$ for all $i \geq 1$ is a $B_h$ set.
\end{lemma}

\begin{proof}
	Assume to the contrary that there are $\tilde{x}_1,\tilde{y}_1,\ldots,\tilde{x}_h,\tilde{y}_h \in \tilde{S}$ satisfying $\{\tilde{x}_1,\ldots,\tilde{x}_h\} \neq \{\tilde{y}_1,\ldots,\tilde{y}_h\}$ and $\tilde{x}_1 + \ldots + \tilde{x}_h = \tilde{y}_1 + \ldots + \tilde{y}_h$. Let $x_i \sim \tilde{x}_i$ and $y_i \sim \tilde{y}_i$ denote the corresponding elements in $S$ for all $1 \leq i \leq h$. We may assume $x_1 \geq \ldots \geq x_h$, $y_1 \geq \ldots \geq y_h$ as well as $x_1 > y_1$ and let $i_1 \geq 1$ such that $x_1 \in I_{i_1}$. By \cref{lemma:Ii}, we have
	\begin{align*}
		|(x_1 + \ldots + x_h) - (y_1 + \ldots + y_h)| \leq h |I_{i_1}| < h 2^{1/\delta} \, i_1^{1/\delta - 1} \leq h 2^{1/\delta} x_1^{1-\delta},
	\end{align*}
	contradicting the fact that $S$ is a $(1-\delta,h 2^{1/\delta})$--strong $B_h$ set.
\end{proof}

We are now ready to prove \cref{thm:transfer} along the lines of the proof of Theorem~12 in~\cite{KohayakawaLeeMoreiraRodl-2019}.

\begin{proof}[Proof of \cref{thm:transfer}]
	Let $S \subset \NN$ be some fixed $(1-\delta,h2^{1 + 1/\delta})$--strong $B_h$ set satisfying $S(n) = n^{u(\delta) + o(1)}$. By \cref{lemma:canonical} we may without loss of generality assume that $|S \cap I_i| \leq 1$ for all $i \geq 1$: observe that we can remove one element for each $I_i$ such that $|I_i \cap S| = 2$ without impacting the growth condition of the resulting set.
	
	Writing $S = \{s_i: i>0 \}$ and for each index $j$, let $i_j$ be such that $s_j \in I_{i_j}$. Let $i_0 = i_0(\delta)$ be as given by \cref{lemma:random_intersecting} and define the random set
	\begin{equation*}
		J = \{ j \in \NN : i_j \geq i_0 \text{ and } R_\delta \cap I_{i_j} \neq \emptyset \}.
	\end{equation*}
	For each $j \in J$, let $\tilde{s}_j$ denote an element in $R_\delta$ satisfying $\tilde{s}_j \sim s_j$. By construction $S^\star = \{\tilde{s}_j : j \in J\}$ is a subset of $R_\delta$ and by \cref{lemma:shifting} it is a $B_h$ set. 
	
	Let us show that, with probability $1$, this set still satisfies $S^\star(n) = n^{u(\delta)+O(1)}$. By Chernoff's bound and \cref{lemma:random_intersecting}, we know that for any $\varepsilon > 0$ there exists some $n_0 = n_0(\varepsilon)$ such that, for any $n \geq n_0$ we have 
	\begin{equation*}
		\PP \big( S^*(n) < n^{u(\delta)-\varepsilon} \big) \leq 2 \exp \big(-n^{u(\delta)-2\varepsilon}/3 \big) \leq 1/n^2.
	\end{equation*}
	Using now Borel--Cantelli Lemma and the fact that $\sum_{n=1}^\infty 1/n^2 < \infty$, it follows that with probability $1$ there exists some $n_1 = n_1(R_{\delta})$ such that $S^{\star}(n) \geq n^{u(\delta)-\varepsilon}$ for any $n \geq n_1$. The desired statement therefore follows.
\end{proof}

\section{Remarks and Open Questions} \label{sec:remarks}

We believe that the lower bound obtained in \cref{thm:aSidon_lower} is a natural extension of the results of Ruzsa and Cilleruelo and that any advance in closing the gap between the upper and lower bounds for strong infinite Sidon sets should probably come as a result of an improvement on either bound in the case of normal 'non-strong' Sidon sets. Unfortunately, this problem has proven surprisingly defiant despite a fair amount of attention.

Concerning the question of the maximum size of Sidon or $B_h$ sets in infinite random sets of integers, \cref{thm:random} currently gives the best lower bound when $h = 2$ and $5/6 < \delta < 1$. In the case of $h > 2$ no other bounds are known, though we believe that using results of Dellamonica et al.~\cite{DellamonicaKohayakawaLeeRodlSamotij-2016,DellamonicaKohayakawaLeeRodlSamotij-2018} for the case of finite random sets, one can establish exact exponents whenever $0 < \delta < h/(2h-1)$. Note that Kohayakawa et al.~\cite{KohayakawaLeeMoreiraRodl-2018} also made use of the case of finite random sets established in \cite{KohayakawaYoshiharuLeeRodlVojtechSamotij-2015}. When $\delta$ is large enough, we believe that \cref{thm:random} again gives the best lower bound that can be obtained without requiring significant new insight into the non-strong and non-random case.

Kohayakawa et al.~\cite{KohayakawaLeeMoreiraRodl-2019} also asked if their upper bound on the size of infinite $\alpha$-strong Sidon sets can be strengthened along the lines of the results of Erd\H{o}s and Tur\'an as well as St\"ohr~\cite{Stoehr-1955} mentioned in the introduction. We extend this question to any $\alpha$-strong $B_h$ set along the lines of \cref{thm:aBh_upper}.

\begin{question}
	Let $0 \leq \alpha < 1$ and $h \geq 2$. Does any $\alpha$-strong $B_h$ set $S$ satisfy
	\begin{equation*}
		\liminf_{n \to \infty} S(n)/n^{(1-\alpha)/h} = 0?
	\end{equation*}
\end{question}

\medskip
\noindent \textbf{Acknowledgements. }	This work was initiated at the annual workshop of the Combinatorics and Graph Theory group at Freie Universität Berlin in September of 2019. The authors would like to thank the institution for enabling this research. Finally, the third author would like to thank Tibor Szabó and the Combinatorics and Graph Theory group at Freie Universität Berlin for their hospitality during a research stay.


\end{document}